\documentclass[11pt]{amsart}
\usepackage{graphicx}
\usepackage[dvips]{epsfig}
\usepackage{pinlabel}
\usepackage{amsmath}
\usepackage{amsfonts}
\usepackage{latexsym}
\usepackage{amssymb}
\usepackage[usenames]{color}
\usepackage{amsthm}
\usepackage[all]{xypic}
\usepackage{enumitem}
\usepackage[breaklinks=true]{hyperref}

\input xy
\xyoption{all}
\def\E{\ifmmode{\mathbb E}\else{$\mathbb E$}\fi} %natural numbers
\def\N{\ifmmode{\mathbb N}\else{$\mathbb N$}\fi} %natural numbers
\def\R{\ifmmode{\mathbb R}\else{$\mathbb R$}\fi} %real numbers
\def\Q{\ifmmode{\mathbb Q}\else{$\mathbb Q$}\fi} %rational numbers
\def\C{\ifmmode{\mathbb C}\else{$\mathbb C$}\fi} %complex numbers
\def\H{\ifmmode{\mathbb H}\else{$\mathbb H$}\fi} %complex numbers
\def\Z{\ifmmode{\mathbb Z}\else{$\mathbb Z$}\fi} %integers
\def\P{\ifmmode{\mathbb P}\else{$\mathbb P$}\fi} %real numbers
\def\T{\ifmmode{\mathbb T}\else{$\mathbb T$}\fi} %real numbers
\def\SS{\ifmmode{\mathbb S}\else{$\mathbb S$}\fi} %real numbers
\def\DD{\ifmmode{\mathbb D}\else{$\mathbb D$}\fi} %real numbers

\newcommand{\ben}{\begin{enumerate}}
\newcommand{\een}{\end{enumerate}}
\newcommand{\be}{\begin{equation}}
\newcommand{\ee}{\end{equation}}
\newcommand{\bea}{\begin{eqnarray}}
\newcommand{\eea}{\end{eqnarray}}
\newcommand{\bc}{\begin{center}}
\newcommand{\ec}{\end{center}}

\newtheorem{thm}{Theorem}[section]
\newtheorem{cor}[thm]{Corollary}
\newtheorem{lem}[thm]{Lemma}
\newtheorem{prop}[thm]{Proposition}

\theoremstyle{definition}
\newtheorem{defn}{Definition}[section]

\theoremstyle{remark}
\newtheorem{rem}{\rm\bfseries{Remark}}[section]

\newtheorem{Conj}[defn]{Conjecture}

\newcommand{\OP}{\operatorname}
\newcommand{\CP}{\ensuremath{\C}P}
\DeclareMathAlphabet{\mathdj}{U}{msb}{m}{n}

\newcommand{\id}{\operatorname{Id}}

\begin{document}

\subjclass[2010]{Primary 53D12; Secondary 53D42}
%==========================================
%% Do not edit the following command
%\setcounter{page}{1}
%\volume{21}
%==========================================

\title[Uniqueness of extremal Lagrangian tori in the disc]{Uniqueness of extremal Lagrangian tori in the four-dimensional disc}

\author[DIMITROGLOU RIZELL]{Georgios Dimitroglou Rizell}
%% \\[.2cm] \small{ Dedicated to somebody }

\thanks{The author is supported by the grant KAW 2013.0321 from the Knut and Alice Wallenberg Foundation. This work was done during a visit of the author to the Institut Mittag-Leffler (Djursholm, Sweden).}

\address{Department of Pure Mathematics and Mathematical Statistics,
Centre for Mathematical Sciences,
University of Cambridge,
Wilberforce Road,
Cambridge,
CB3 0WB, United Kingdom.}
\email{g.dimitroglou@maths.cam.ac.uk}

\begin{abstract}
The following interesting quantity was introduced by K. Cieliebak and K. Mohnke for a Lagrangian submanifold $L$ of a symplectic manifold: the minimal positive symplectic area of a disc with boundary on $L$. They also showed that this quantity is bounded from above by $\pi/n$ for a Lagrangian torus inside the $2n$-dimensional unit disc equipped with the standard symplectic form. A Lagrangian torus for which this upper bound is attained is called extremal. We show that all extremal Lagrangian tori inside the four-dimensional unit disc are contained in the boundary $\partial D^4=S^3$. It also follows that all such tori are Hamiltonian isotopic to the product torus $S^1_{1/\sqrt{2}} \times S^1_{1/\sqrt{2}} \subset S^3$. This provides an answer to a question by L. Lazzarini in the four-dimensional case.
\end{abstract}

\keywords{Capacities, Extremal Lagrangian tori, monotone Lagrangian tori}

\maketitle

\section{Introduction and results}
In the following we will consider the standard even dimensional symplectic vector space $(\C^n,\omega_0:=dx_1\wedge dy_1+\hdots+dx_n \wedge dy_n)$, as well as the projective space $(\CP^n,\omega_{\OP{FS},r})$ endowed with the Fubini-Study symplectic two-form. We here normalise $\omega_{\OP{FS},r}$ so that a line $\ell \subset \CP^n$ has symplectic area equal to $\int_\ell \omega_{\OP{FS},r}=\pi r^2$. We also write $\omega_{\OP{FS}}:=\omega_{\OP{FS},1/\sqrt{\pi}}$. See Section \ref{sec:prel} for more details.

Neck-stretching techniques were successfully used in \cite{PuncturedHolomorphic} by K. Cieliebak and K. Mohnke in order to prove the Audin conjecture, first formulated in \cite{FibresNormaux} by M. Audin: Every Lagrangian torus in $\C^n$ or $\CP^n$ bounds a disc of positive symplectic area and Maslov index equal to two. The same techniques were also used to deduce properties concerning the following quantity for a Lagrangian submanifold, which was introduced in the same article. (We here restrict our attention to Lagrangian tori.) Given a Lagrangian torus $L \subset (X,\omega)$ inside an arbitrary symplectic manifold, we define
\[ A_{\OP{min}}(L):=\inf_{A \in \pi_2(X,L) \atop \int_A \omega > 0} \int_A \omega \in [0,+\infty].\]
This quantity can then be used in order to define a capacity for the symplectic manifold $(X,\omega)$ as follows:
\[ c_{\OP{Lag}}(X,\omega) := \sup_{L\subset (X,\omega) \text{ Lag. torus}} A_{\OP{min}}(L) \in [0,+\infty].\]
We refer to \cite{PuncturedHolomorphic} for the properties satisfied by this capacity. In view of this it is natural to consider:
\begin{defn}[\cite{PuncturedHolomorphic}] A Lagrangian torus $L \subset (X,\omega)$ satisfying
\[A_{\OP{min}}(L)=c_{\OP{Lag}}(X,\omega) \]
is called \emph{extremal}.
\end{defn}

The above capacity has been computed only for a limited number of symplectic manifolds, notably:
\begin{thm}[Theorem 1.1 and Corollary 1.3 in \cite{PuncturedHolomorphic}]
\label{thm:cap}
We have
\begin{eqnarray} \label{1} & & c_{\OP{Lag}}(B^{2n},\omega_0)  =  \pi/n,\\
\label{2} & & c_{\OP{Lag}}(\CP^n,\omega_{\OP{FS},r}) = r^2\pi/(n+1),
\end{eqnarray}
and in particular $c_{\OP{Lag}}(D^{2n},\omega_0)  =  \pi/n$.
\end{thm}A straight-forward calculation shows that the $n$-dimensional Clifford torus
\[L_{\OP{Cl}}:=\left(S^1_{\frac{1}{\sqrt{n}}}\right)^n \subset S^{2n-1} = \partial D^{2n} \subset (\C^n,\omega_0),\]
contained inside the boundary of the $2n$-dimensional unit disc is extremal. In the case when $n=1$, the Clifford torus is clearly the only extremal Lagrangian torus. Furthermore, a monotone Lagrangian torus $L \subset (\CP^n,\omega_{\OP{FS}})$ is extremal, as follows by elementary topological considerations together with the fact that there exists a representative of $\pi_2(\CP^2,L)$ having Maslov index two and positive symplectic area by \cite[Theorems 1.1, 1.2]{PuncturedHolomorphic}. (For previous related results, consider \cite{NewObstruction}, \cite{Polterovich:MaslovClass}, \cite{Oh:spectral}, \cite{FloerAnomalyI}, \cite{Buhovsky:MaslovClass}, and \cite{FloerHomUniv}.)

In \cite{PuncturedHolomorphic} the author learned about the following two conjectures, the first one originally due to L. Lazzarini:
\begin{Conj}
\label{conj:1} All extremal Lagrangian tori $L \subset (D^{2n},\omega_0)$ are contained inside the boundary $\partial D^{2n}=S^{2n-1}$.
\end{Conj}
\begin{Conj}
\label{conj:2} All extremal Lagrangian tori $L \subset (\CP^n,\omega_{\OP{FS}})$ are monotone.
\end{Conj}

Our main result is a positive answer to Conjecture \ref{conj:1} in dimension four.
\begin{thm}
\label{thm:main}
All extremal Lagrangian tori $L \subset (D^4,\omega_0)$ are contained inside the boundary, i.e.~$L \subset S^3 = \partial D^4$.
\end{thm}
After a consideration of the possible Lagrangian tori inside the three-dimensional unit sphere using classical techniques, we also obtain the following classification result.
\begin{cor}
\label{cor:main}
All extremal Lagrangian tori $L \subset (D^4,\omega_0)$ are isotopic to the Clifford torus $S^1_{1/\sqrt{2}} \times S^1_{1/\sqrt{2}} \subset S^3 = \partial D^4$ by a Hamiltonian isotopy of $(D^4,\omega_0)$ preserving the boundary set-wise.
\end{cor}

The proof of our main result Theorem \ref{thm:main} consists of analysing the pseudoholomorphic curves produced in Cieliebak-Mohnke's proof of Theorem \ref{thm:cap} from \cite{PuncturedHolomorphic}. Their proof is based upon a pseudoholomorphic curve technique called the ``splitting construction'' or ``stretching the neck'', which first appeared in the setting of symplectic field theory in the work \cite{IntroSFT} by Y. Eliashberg, A. Givental, and H. Hofer. Pseudoholomorphic curves were introduced by M. Gromov \cite{Gromov}. We will restrict our attention to real four-dimensional symplectic manifolds. In this setting pseudoholomorphic curves behave particularly well, since one can apply techniques such as positivity of intersection due to D. McDuff \cite{LocalCurve}, and automatic transversality which in the present setting is due to C. Wendl \cite{Auttrans}. These four-dimensional techniques are crucial to our proof, and it is not clear to the author if the argument can be modified to work in arbitrary dimensions.

The so-called splitting construction involves studying the limit of pseudoholomorphic curves under a sequence $J^\tau$, $\tau \ge 0$, of tame almost complex structures on $(X,\omega)$ which stretches the neck around a hypersurface $Y \subset (X,\omega)$ of contact type. Loosely speaking, such a sequence introduces a neck of the form $[-(\tau+\epsilon),\tau+\epsilon] \times Y \hookrightarrow X$ in a neighbourhood of $Y$, where the almost complex structure is cylindrical. In our case, the hypersurface $Y \subset X= \CP^n$ of contact type will be taken to be the unit cotangent bundle of the Lagrangian torus. The compactness results for sequences of $J^\tau$-holomorphic curves as $\tau \to +\infty$ were obtained in \cite{CompSFT} by F. Bourgeois, Y. Eliashberg, C. Wysocki, and E. Zehnder, and independently in \cite{CompactnessPunctured} by K. Cieliebak and K. Mohnke. Due to the fact that tori admit flat metrics which, moreover, induce foliations by families of closed geodesics, the limit of pseudoholomorphic curves under a stretching of the neck behave particularly well in this case. In addition to \cite{PuncturedHolomorphic}, Lagrangian tori have previously been studied using the splitting construction in a series of work; among others, see \cite{Unlinking} by the author together with J. D. Evans, \cite{Polydisks} by R. Hind and S. Lisi, and \cite{LagIsoTori} by the author together with E. Goodman and A. Ivrii.

\section{Preliminaries}
\label{sec:prel}
This paper concerns Lagrangian tori inside the open unit ball and closed unit disc
\[B^{2n} \subset D^{2n} \subset (\C^n,\omega_0=dx_1\wedge dy_1 + \hdots + dx_n \wedge dy_n)\]
endowed with the standard symplectic two-form $\omega_0$, as well as the projective space $(\CP^n,\omega_{\OP{FS}})$ endowed with the Fubini-Study symplectic two-form. Recall that a half-dimensional submanifold of a symplectic manifold is said to be {\bf Lagrangian}, if the symplectic form vanishes on its tangent space. It will be useful to compactify the open ball $B^{2n}_r \subset (\C^n,\omega_0)$ of radius $r>0$ to the projective plane
\[(\CP^n,\omega_{\OP{FS}_r}) \supset (\CP^n \setminus D_\infty,\omega_{\OP{FS},r}) = (B^{2n}_r,\omega_0)\]
endowed with the Fubini-Study symplectic form, where $D_\infty$ denotes the divisor at infinity. The Fubini-Study form $\omega_{\OP{FS},r}$ is here normalised so that a surface $\ell \subset \CP^n$ of degree one, i.e.~$[\ell]\in H_2(\CP^n)$ is the generator of positive symplectic area, satisfies $\int_\ell \omega_{\OP{FS},r}=\pi r^2$. We also write $\omega_{\OP{FS}}:=\omega_{\OP{FS},1/\sqrt{\pi}}$.

The above symplectic manifolds $(X,\omega)$ are {\bf monotone}: the first Chern class satisfies $c_1(X,\omega) = \kappa [\omega] \in H^2(X;\R)$ for some $\kappa >0$. More precisely, $\kappa >0$ can be chosen arbitrarily for $(\C^n,\omega_0)$, while $\kappa=(n+1)/(\pi r^2)$ for $(X,\omega)=(\CP^n,\omega_{\OP{FS},r})$. Recall that a Lagrangian torus inside $(X,\omega)$ is said to be {\bf monotone} provided that its Maslov class satisfies $\mu_L=2\kappa[\omega] \in H^2(X,L;\R)$ for a number $\kappa>0$ as above. Pseudoholomorphic curve techniques work particularly well for monotone Lagrangian tori, and they are known to satisfy many rigidity properties.

Recall that a {\bf tame almost complex structure} $J$ on a symplectic manifold $(X,\omega)$ is an endomorphism $J \in \OP{End}(TX)$ satisfying $J^2 =-\id_{TX}$ together with the property that $\omega(v,Jv)>0$ whenever $v \neq 0$. The space of tame almost complex structures is contractible by \cite{Gromov}. In the latter article Gromov also established his celebrated compactness theorem for pseudoholomorphic curves for a tame almost complex structure. Recall that, given a choice of almost complex structure $J$ and a Riemann surface $(\Sigma,i)$, a map
\[ u \colon (\Sigma,i) \to (X,J) \]
is {\bf pseudoholomorphic} given that the fully non-linear Cauchy-Riemann type equation $J \circ du = du \circ i$ is satisfied.

The present article will mainly consider punctured pseudoholomorphic spheres in non-compact symplectic manifolds. More precisely, outside of an open pre-compact sub-domain with smooth boundary, the symplectic manifold will consist of convex and concave cylindrical ends symplectomorphic to half symplectisations
\begin{gather*}
[0,+\infty) \times Y_+ \subset (\R \times Y_+,d(e^t\alpha_+)), \\
(-\infty,0] \times Y_- \subset (\R \times Y_-,d(e^t\alpha_-)),
\end{gather*}
respectively, where $t$ is a coordinate on the $\R$-factor, and $\alpha_\pm$ are contact one-forms on the closed manifolds $Y_\pm$. Recall the definition of the {\bf Reeb vector field} $R \in \Gamma(TY)$ on a contact manifold $(Y,\alpha)$ with contact form $\alpha$, which is determined by the equations
\[\alpha(R)=1, \:\: d\alpha(R,\cdot)=0.\]
Following \cite{CompSFT}, a tame almost complex structure $J$ on $(\R \times Y,d(e^t\alpha))$ is said to be {\bf cylindrical} given that:
\begin{itemize}
\item $J$ is invariant under translations of the $t$ coordinate;
\item $J\partial_t=R$; and
\item $J\xi =\xi$, where $\xi:=\ker \alpha \subset TY$.
\end{itemize}
The aforementioned article extends Gromov's compactness theorem to the space of pseudoholomorphic curves for tame almost complex structures that are cylindrical outside of a compact subset, given that these curves have a uniform bound on their Hofer energy (see \cite{CompSFT} for the definition).

By a {\bf finite energy curve} we will mean a pseudoholomorphic curve of finite (Hofer) energy. We will not give the formal definition, but we point out that a proper punctured pseudoholomorphic curve
\[ u \colon (\Sigma \setminus \{ p_1,\hdots,p_m\},i) \to (X,\omega),\]
where $(\Sigma,i)$ is a closed Riemann surface and $(X,\omega)$ has non-compact cylindrical ends, is a finite energy curve if and only if $u$ is asymptotic to cylinders $\R \times \gamma_i \subset \R \times Y_\pm$ contained in the cylindrical ends of $X$ at each of its punctures $p_i$, $i=1,\hdots,m$. Here $\gamma_i \subset Y_\pm$ are periodic integral curves of the Reeb vector field $R_\pm$ -- usually called {\bf periodic Reeb orbits}.

\section{The proof of Corollary \ref{cor:main}}
Since the torus $L \subset S^3 \subset B^4$ is Lagrangian, it is foliated by integral curves of the characteristic foliation
\[\ker \omega_0 |_{TS^3} \subset TS^3 \subset TB^4.\]
These integral curves are the periodic Reeb orbits
\[S^1=\R/2\pi\Z \ni \theta \mapsto e^{i\theta}\mathbf{z} \in S^3, \:\: \mathbf{z} \in S^3 \subset \C^2,\]
of the standard contact form $\alpha_{\OP{std}}:=\frac{1}{2}\sum_{i=1}^2(x_idy_i-y_idx_i)$ on $S^3$. The foliation of the sphere by these Reeb orbits induces the Hopf fibration
\begin{gather*}
p \colon S^3 \to \CP^1,\\
(z_1,z_2) \mapsto [z_1:z_2],
\end{gather*}
with the above periodic Reeb orbits as $S^1$-fibres, while the base is given by the orbit space diffeomorphic to $\CP^1$. Recall that the latter space is endowed with the canonical symplectic form $\omega_{\OP{FS},1}$ obtained from the corresponding symplectic reduction $(\C^2,\omega_0) \supset S^3 \xrightarrow{p} (\CP^1,\omega_{\OP{FS},1})$.

Since $L$ is foliated by periodic Reeb orbits of the same period, the Reeb flow on $S^3$ (i.e.~multiplication by $e^{i2t}$) induces a smooth and free $S^1$-action on $L$. In other words, we may consider $L$ as an $S^1$-bundle $q \colon L \to L/S^1 \cong S^1$, with base given as the quotient of $L$ by this action. By topological reasons this must be a trivial $S^1$-bundle over $S^1$.

In particular, for any choice of section $\sigma \colon L/S^1 \to L$ of $q$, the projection $p \circ \sigma$ is an embedded closed curve $\gamma := p \circ \sigma \colon S^1 \hookrightarrow (\CP^1,\omega_{\OP{FS},1})$. Using the assumption that $L$ is extremal, we conclude that the two smooth discs $\CP^1 \setminus \gamma(S^1)$ each must be of symplectic area equal to $\pi/2$. Otherwise, one could readily lift the disc of smaller area to a disc inside $S^3 \to \CP^1$ having boundary on $L$ and being of the same area, thus contradicting the assumption that $L$ is extremal.

Using elementary methods one can construct a Hamiltonian isotopy
\[\phi^s_{H_s} \colon (\CP^1,\omega_{\OP{FS,1}}) \to (\CP^1,\omega_{\OP{FS,1}}),\]
induced by the time-dependent Hamiltonian $H_s \colon \CP^1 \to \R$, taking the curve $\gamma$ to the equator
\begin{gather*}
\gamma_0 \colon S^1 =\R/2\pi\Z \to \CP^1,\\
\theta \mapsto p(e^{i\theta/2},e^{-i\theta/2}).
\end{gather*}
Observe that $p^{-1}(\gamma_0)=S^1_{1/\sqrt{2}} \times S^1_{1/\sqrt{2}} \subset S^3$ is the sought Clifford torus.

The Hamiltonian isotopy $\phi^s_{H_s}$ lifts to a \emph{contact-form preserving} isotopy $\phi^s \colon (S^3,\alpha_{\OP{std}}) \to (S^3,\alpha_{\OP{std}})$ induced by the contact Hamiltonian $H_s \circ p \colon S^3 \to \R$, i.e.~$p\circ \phi^s=\phi^s_{H_s}$.

Since the contactomorphisms $\phi^s$ preserve the contact form $\alpha_{\OP{std}}$ it follows that
\begin{gather*}
(\R \times S^3,d(e^t\alpha_{\OP{std}})) \to (\R \times S^3,d(e^t\alpha_{\OP{std}})), \\
(t,y) \mapsto (t,\phi^s(y)),
\end{gather*}
defines a one-parameter family of symplectomorphisms which, moreover, is generated by the time-dependent Hamiltonian of the form $e^t H_s \circ p \colon \R \times S^3 \to \R$.

Utilising the symplectic identification
\begin{gather*}
(\C^2 \setminus \{ 0\},\omega_0) \to (\R \times S^3,d(e^t\alpha_{\OP{std}})), \\
\mathbf{z} \mapsto (2\log \|\mathbf{z}\|,\mathbf{z}/\|\mathbf{z}\|),
\end{gather*}
we obtain a time-dependent Hamiltonian
\begin{gather*}
\widetilde{H}_s \colon \C^2 \setminus \{0\} \to \R, \\
\mathbf{z} \mapsto \|\mathbf{z}\|^2 \cdot H_s \circ p(\mathbf{z}),
\end{gather*}
whose induced flow on $\C^2 \setminus \{0\}$ preserves each concentric sphere $S^3_r$, $r>0$, while it satisfies the property that $p \circ \phi^s_{\widetilde{H}_s} = \phi^s_{H_s}$. After a smoothing of $\widetilde{H}_s$ in a small neighbourhood of the origin $0 \in \C^2$, we have finally produced our sought Hamiltonian isotopy of $(D^4,\omega_0)$.

\section{The proof of Theorem \ref{thm:main}}
\label{sec:proof}
Our proof follows from the techniques in \cite{PuncturedHolomorphic} that are used to prove part \eqref{1} of Theorem \ref{thm:cap}, combined with the automatic transversality result from \cite{Auttrans}. Observe that the latter theory only is applicable to four-dimensional symplectic manifolds.

By contradiction we assume that $\varphi \colon \T^2 \hookrightarrow (D^4,\omega_0)$ is a fixed Lagrangian torus $L$ which is extremal, but not contained entirely in the boundary $S^3 =\partial D^4$. Here we make the identification $\T^2=S^1 \times S^1=\R/2\pi\Z \times \R / 2\pi\Z$, with the induced circle-valued coordinates $\theta_1,\theta_2$.

We start by fixing a so-called Weinstein neighbourhood of $L$ (see e.g.~\cite{SympTop}), which is an extension of $\varphi$ to a symplectomorphism
\begin{gather*}
\Phi \colon (T^*_{3\epsilon_0}\T^2,d\lambda_{\T^2}) \to (\C^2,\omega_0),\\
\Phi|_{\T^2}=\varphi,
\end{gather*}
for some fixed $\epsilon_0 >0$. Here $T^*_{s}\T^2:=\T^2 \times B^2_s \subset \T^2 \times \R^2$ is the co-disc bundle of radius $s>0$ for the flat torus, and where the Liouville one-form is given by $\lambda_{\T^2}=p_1d\theta_1+p_2d\theta_2$ for the standard coordinates $(p_1,p_2)$ on the $\R^2$-factor. Writing $S^*_s\T^2 := \T^2 \times S^1_s \subset \T^2 \times \R^2$ for the corresponding co-sphere bundle, we will once and for all fix a single fibre $F_p$ of $S^*_{2\epsilon_0}\T^2$ above a point $p \in \T^2 \subset \Phi^{-1}(B^4)$. After making an appropriate choice of $p \in \T^2$, and possibly after replacing $\epsilon_0>0$ with a sufficiently small number, we may assume that
\[S_0:= \Phi(F_p) \subset B^4 \setminus L = D^4 \setminus (\partial D^4 \cup L).\]
Here we have obviously used the assumption that $L$ is not contained entirely inside $\partial D^4$. See Figure \ref{fig:fibre} for a schematic picture.

\begin{figure}[htp]
\begin{center}
\vspace{3mm}
\labellist
\pinlabel $T^*_\epsilon L$ at 51 38
\pinlabel $U_0$ at 49 90
\pinlabel $\color{red}S_0$ at 69 87
\pinlabel $L$ at 106 75
\pinlabel $B^4_{1+\delta}$ at 144 21
\pinlabel $p$ at 90 51
\pinlabel $B^4$ at 74 20
\endlabellist
\includegraphics{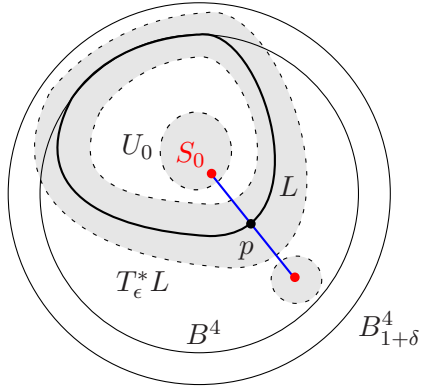}
\caption{The embedded circle $S_0 \subset B^4$ being the image of a fibre $F_p \subset S^*_{2\epsilon_0}L$ over $p \in L$ and which non-trivially links the Lagrangian torus $L$. The co-disc bundle $T^*_\epsilon L$ is disjoint from the neighbourhood $U_0$ of this fibre, and is contained inside $B^4_{1+\delta}$, given that $0<\epsilon<\epsilon_0$ is sufficiently small.}
\label{fig:fibre}
\end{center}
\end{figure}

The following property will be crucial to us. We fix an open neighbourhood $U_0$ of $S_0$ whose closure satisfies
\[ \overline{U_0} \subset B^4 \setminus \Phi(T^*_{\epsilon_0}\T^2). \]
We also fix a compatible almost complex structure $J_{U_0}$ on $(U_0,\omega_0)$ which can be extended to a compatible almost complex structure on $(\C^2,\omega_0)$. The monotonicity property for pseudoholomorphic curves \cite[Proposition 4.3.1(ii)]{SomeProp} implies that
\begin{enumerate}[label=(M), ref=(M)]
\item \label{M} There exists a constant $\hbar>0$ for which the following holds: Any proper $J_{U_0}$-holomorphic curve $C \subset U_0$ satisfying $C \cap S_0 \neq \emptyset$ has symplectic area bounded from below by $\int_C \omega_0 \ge \hbar$.
\end{enumerate}

\subsection{The proof of part \eqref{1} of Theorem \ref{thm:cap} in dimension four}
We here reprove \eqref{1} in the four-dimensional case, which originally was established in \cite{PuncturedHolomorphic} by Cieliebak-Mohnke in arbitrary dimensions. We follow exactly the same ideas, with the only difference that, instead of applying \cite[Corollary 3.3]{PuncturedHolomorphic} which holds in the nondegenerate setting, we adapt this argument to a setting which is degenerate in the Morse-Bott sense. Theorem \ref{thm:main} will then be seen to follow after a more careful investigation of the involved pseudoholomorphic curves; this is done in Section \ref{sec:further} below.

For each $\delta>0$, we consider the symplectic embeddings
\[(B^4,\omega_0) \subset (B^4_{1+\delta},\omega_0) = (\CP^2 \setminus D_\infty,\omega_{\OP{FS},1+\delta})\subset (\CP^2,\omega_{\OP{FS},1+\delta}),\]
together with a number $0<\epsilon<\epsilon_0$ depending on $\delta>0$ for which $\Phi(T^*_{\epsilon}\T^2) \subset B^4_{1+\delta/2}$.

\subsubsection{The Morse-Bott contact form on $S^*_1L$}
First we define the tame almost complex structure $J_0$ on $T^*\T^2$ determined by
\[J_0(\partial_{\theta_i})=-\rho(\| \mathbf{p}\|)\partial_{p_i}, \:\:i=1,2,\]
where $\rho \colon \R_{\ge 0} \to \R_{\ge 0}$ satisfies $\rho'(t) \ge 0$, $\rho(t) \equiv \epsilon/4$ for $t \le \epsilon/4$, and $\rho(t)=t$ for $t \ge \epsilon/3$. We also define the tame almost complex structure $J_{\OP{cyl}}$ on $T^*\T^2 \setminus 0_{\T^2}$ determined by
\[ J_{\OP{cyl}}(\partial_{\theta_i})=-\| \mathbf{p}\|\partial_{p_i}, \:\:i=1,2,\]
which hence coincides with $J_0$ in the subset $\{ \|\mathbf{p}\| \ge \epsilon/3\}$. Note that the latter almost complex structure is \emph{cylindrical} in the sense of \cite{CompSFT}, given that we use the exact symplectic identification of $(T^*\T^2 \setminus 0_{\T^2},d\lambda_{\T^2})$ with the symplectisation
\[ (\R \times S^*_1\T^2,d(e^t\alpha_0)), \:\: \alpha_0:=\lambda_{\T^2}|_{T(S^*_1\T^2)},\]
which sends the level set $\{\| \mathbf{p}\|=c\}$ to the level set $\{ t= \log c\}$.

We observe that the Reeb flow on $(S^*_1L,\alpha_0)$ coincides with the so-called cogeodesic flow for the canonical flat metric on $L \cong \T^2=(\R/2\pi\Z)^2$. In particular, it follows that the periodic Reeb orbits come in manifolds diffeomorphic to $S^1$: there is one such family $\Gamma_\eta \simeq S^1$ of periodic Reeb orbits for each non-zero homology class $\eta \in H_1(L) \setminus \{0\}$. These manifolds of orbits are moreover non-degenerate in the Morse-Bott sense; see \cite{BourgeoisBott} for more details.

\subsubsection{A sequence of almost complex structures stretching the neck}
\label{sec:neckstretch}

Consider a sequence of almost complex structures $J^\tau$, $\tau \ge 0$, which satisfy the following properties:
\begin{itemize}
\item Inside $\CP^2 \setminus \Phi(T^*_{\epsilon/2}\T^2)$, the family $J^\tau$ is independent of $\tau$. Furthermore, we require that:
\begin{itemize}
\item In the subset $U_0 \subset B^4 \setminus \Phi(T^*_{\epsilon}\T^2)$ we have $J^\tau|_{U_0}=J_{U_0}$.
\item Inside $\CP^2 \setminus B^4_{1+\delta/2}$ we have $J^\tau=i$, where the $i$ denotes the standard integrable complex structure on $(\CP^2,\omega_{\OP{FS},1+\delta})$. (In particular, the line $D_\infty$ at infinity is $J^\tau$-holomorphic for each $\tau \ge 0$.)
\end{itemize}
\item In the subset $\Phi(T^*_{\epsilon/3}\T^2)$, each $J^\tau$ is the push-forward of $J_0$ under the map $\Phi$.
\item In the subset $\Phi(T^*_{\epsilon/2}\T^2) \setminus \Phi(T^*_{\epsilon/3}\T^2)$, identified with
\[ ([\log{\epsilon/3},\log{\epsilon/2}) \times S^*_1\T^2,d(e^t\alpha_0)),\]
the almost complex structure $J^\tau$, $\tau \ge 0$, is the push-forward of $J_{\OP{cyl}}$ under the (non-symplectic!) identification
\[ [\log{\epsilon/3},\log{\epsilon/2}+\tau) \times S^*_1\T^2 \cong [\log{\epsilon/3},\log{\epsilon/2}) \times S^*_1\T^2,\]
induced by a diffeomorphism $[\log{\epsilon/3},\log{\epsilon/2}) \cong [\log{\epsilon/3},\log{\epsilon/2}+\tau)$.
\end{itemize}
The above sequence $J^\tau$, $\tau \ge 0$, of tame almost complex structures is said to ``stretch the neck'' around the hypersurface
\[\Phi(S^*_{\epsilon/2}\T^2) \subset (\CP^2,\omega_{\OP{FS},1+\delta})\]
of contact type as $\tau \to +\infty$, where this hypersurface has been endowed with the contact form $\alpha_0$ described above.

We also need to specify the tame almost complex structure $J^\infty$ on the symplectic manifold $\CP^2 \setminus L$ with a concave cylindrical end, which is determined by
\begin{itemize}
\item $J^\infty=J^\tau=J^0$ inside $\CP^2 \setminus \Phi(T^*_{\epsilon/2}\T^2)$; and
\item $J^\infty$ is the push-forward of $J_{\OP{cyl}}$ under $\Phi$ in the subset $\Phi(T^*_{\epsilon/2}\T^2 \setminus 0_{\T^2}) \subset \CP^2 \setminus L$.
\end{itemize}
Again, observe that the line at infinity $D_\infty \subset \CP^2$ is $J^\infty$-holomorphic by construction.

\subsubsection{The existence of pseudoholomorphic buildings of degree one}
The following existence result for pseudo-holomorphic curves in $(\CP^2,\omega_{\OP{FS},r})$ due to Gromov is crucial.
\begin{thm}[\cite{Gromov}]
\label{thm:gromov}
Let $J$ be an arbitrary tame almost complex structure on $(\CP^2,\omega_{\OP{FS},r})$. There exists a unique $J$-holomorphic sphere of degree one which satisfies either of the following:
\begin{itemize}
\item A point constraint at two different points $p,q \in \CP^2$; or
\item A tangency condition to $\C\mathbf{v} \subset T_p\CP^2$ at a given point $p \in \CP^2$, where $\mathbf{v} \neq 0$.
\end{itemize}
Furthermore, this sphere is embedded, has Fredholm index 4, and is transversely cut out.
\end{thm}

Consider a sequence of $J^{\tau_i}$-holomorphic spheres $\ell_{\tau_i} \subset \CP^2$ of degree one, for which $\lim_{i \to \infty} \tau_i=\infty$, and where we have used a neck-stretching sequence of tame almost complex structures as described in Section \ref{sec:neckstretch}. Gromov's compactness theorem extended to the SFT setting, more precisely \cite[Theorem 10.3]{CompSFT}, or alternatively \cite[Theorem 1.2]{CompactnessPunctured}, now gives the following. After passing to a subsequence, this sequence of parametrised spheres converges to a ``pseudoholomorphic building'', where we refer to the latter papers for a description of the relevant topology. By a pseudoholomorphic building, we mean a collection of parametrised punctured spheres of the following form:
\begin{itemize}
\item A {\bf top level} consisting of a finite number of punctured $J^\infty$-holomorphic spheres $A_1,\hdots,A_k \subset \CP^2 \setminus L$;
\item A (possibly zero) number of {\bf middle levels} consisting of a finite number of punctured $J_{\OP{cyl}}$-holomorphic spheres $B_1,\hdots,B_l \subset \R \times S^*_1L$; and
\item A (possibly empty) {\bf bottom level} consisting of a finite number of punctured $J_0$-holomorphic spheres $C_1, \hdots, C_m \subset T^*L$.
\end{itemize}
All of the above punctured spheres are of finite energy, and the asymptotic orbits match in order for the different levels to topologically glue to form a cycle in $\CP^2$ which is a sphere of degree one. Again, we refer to the papers above for more details.

Note that we often abuse notation and suppress the parametrisation from the notation of a pseudoholomorphic curve, as well as from the notation of a component involved in a pseudoholomorphic building.

A one-punctured pseudoholomorphic sphere will be referred to as a {\bf pseudoholomorphic plane}, while a two-punctured pseudoholomorphic sphere will be referred to as a {\bf pseudoholomorphic cylinder}. Observe that the middle levels contain pseudoholomorphic cylinders of the form $\R \times \gamma \subset \R \times S^*_1L$, where $\gamma \subset (S^*_1L,\alpha_0)$ is a periodic Reeb orbit. The latter cylinders will be called {\bf trivial cylinders}. By the SFT compactness theorem, every non-empty middle level must contain at least one punctured sphere which is not a trivial cylinder.

We also state the following simple, but useful, lemma.
\begin{lem}
\label{lem:noplanes}
For the almost complex structures $J_0$ and $J_{\OP{cyl}}$ on $T^*L$ and $\R \times S^*_1L$, respectively, there are no non-constant pseudoholomorphic planes of finite energy.
\end{lem}
\begin{proof}
There are no contractible geodesics on $L$ for the flat metric, and an appropriate compactification of such a plane would produce a null-homology of the geodesic to which it is asymptotic.
\end{proof}

\subsubsection{The heart of the proof: producing a pseudoholomorphic building containing three planes}

We are now ready to state the main result in this subsection, from which part \eqref{1} of Theorem \ref{thm:cap} can be seen to follow. We follow the method of \cite[Section 4]{PuncturedHolomorphic}, but applied in the current Morse-Bott setting. Pick a generic tangency condition $\C\mathbf{v} \subset T_p \CP^2$ for $p \in L$, and consider the sequence of $J^\tau$-holomorphic spheres of degree one satisfying this condition with $\tau \to +\infty$ (they exist by Gromov's Theorem \ref{thm:gromov}). After passing to a subsequence, the SFT compactness result produces a limit pseudoholomorphic building containing one component $C_0 \subset T^*L$ passing through the point $p \in L$ and satisfying the same tangency, when considered as a parametrised curve. We observe that, by definition, any given tangency condition is satisfied at a singular point of a parametrised pseudoholomorphic curve (e.g.~the image of a branched point).
\begin{prop}
\label{prop:three}
After a perturbation of $J_0$ inside an arbitrarily small neighbourhood of $p \in T^*L$, the limit pseudoholomorphic building produced above consists of at least two $J^\infty$-holomorphic planes $A_1,A_2 \subset \CP^2 \setminus L$ that are disjoint from the line at infinity $D_\infty$. Moreover, the planes $A_1$ and $A_2$ are connected to the unique component $A_\infty \subset \CP^2 \setminus L$ of the building passing through $D_\infty$ via the component $C_0 \subset T^*L$. See Figure \ref{fig:building} for a schematic picture.
\end{prop}
\begin{figure}[htp]
\begin{center}
\vspace{3mm}
\labellist
\pinlabel $\CP^2\setminus L$ at -25 58
\pinlabel $T^*L$ at -14 20
\pinlabel $\color{blue}D_\infty$ at 32 82
\pinlabel $\color{blue}D_\infty$ at 220 82
%\pinlabel $1$ at 31 54
\pinlabel $A_\infty$ at 51 68
%\pinlabel $1$ at 112 54
%\pinlabel $1$ at 72 54
\pinlabel $A_1$ at 82 68
\pinlabel $A_2$ at 123 68
\pinlabel $C_0$ at 125 10
\pinlabel $p$ at 64 9
\pinlabel $p$ at 188 9
\pinlabel $\mathbf{v}$ at 78 21
\pinlabel $\mathbf{v}$ at 202 21
\endlabellist
\includegraphics{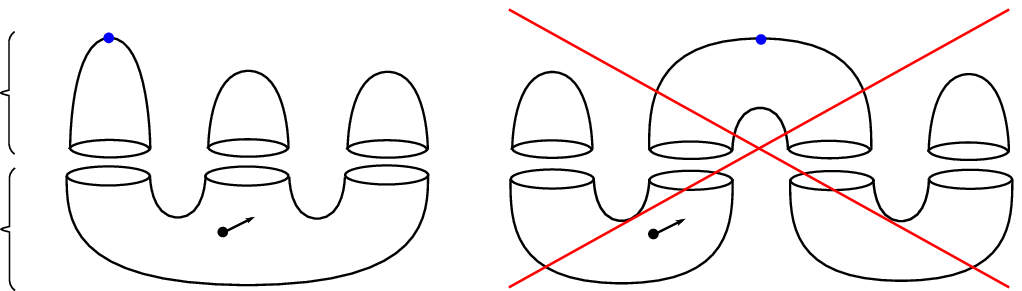}
\caption{Proposition \ref{prop:three} produces a pseudoholomorphic building as shown on the left. For a general Lagrangian torus $L$ it is possible that the planes $A_1$, $A_2$ are connected to $C_0$ via intermediate components, and that $A_\infty$ itself is a building (i.e.~a broken plane).}
\label{fig:building}
\end{center}
\end{figure}
\begin{proof}
By Lemma \ref{lem:noplanes} the component $C_0 \subset T^*L$ passing through $p \in L$ has at least two punctures. We will show that, in fact, it must have at least three punctures. The proposition readily follows from the latter statement. Namely, first observe that precisely one component in $\CP^2 \setminus L$ of the building intersects $D_\infty$ by positivity of intersection together with $[D_\infty] \bullet [D_\infty]=1$ (recall that $D_\infty$ is $J^\infty$-holomorphic by assumption). Using Lemma \ref{lem:noplanes} together with a topological consideration one now deduces the existence of the two sought planes. What remains is thus to prove that $C_0$ has at least three punctures.

A pseudoholomorphic cylinder $C \subset T^*L$ has Fredholm index given by
\[ \OP{index}(C)= (\OP{RS}(\Gamma_1)+1/2)+(\OP{RS}(\Gamma_2)+1/2)\]
as follows from \cite[Formula (3)]{PuncturedHolomorphic}. Here $\Gamma_1$ and $\Gamma_2$ are the families of periodic Reeb orbits containing the asymptotic orbits of $C$, and $\OP{RS}(\Gamma_i)$ denotes the Robbin-Salamon index defined in \cite[Remark 5.4]{MaslovIndex}. Here the latter index has been computed using the canonical trivialisation of the complex determinant bundle of $T(T^*L)$. In this setting, the Robbin-Salamon index can be related to the Morse index $\iota_\mu$ and nullity $\iota_\nu$ of the corresponding geodesics on $L$ by the formula
\[ \OP{RS}(\Gamma_i) = \iota_\mu+(1/2)\iota_\nu=1/2\] 
in \cite[Equation 60]{CieFra}. Here we have used the fact that $\i_\mu=0$ and $i_\nu=0$ for all closed geodesics on the flat $\T^2$; they are all of minimal length in their homology class, and come in one-dimensional families. In conclusion, we have shown that
\[ \OP{index}(C)= 2.\]

After a perturbation of $J_0$ supported in an arbitrarily small neighbourhood of $p \in L$, we may assume that every simply covered pseudoholomorphic cylinder in $T^*L$ is transversely cut out by a standard transversality argument; see \cite[Section 3.4]{JHolCurves}. It thus follows that the moduli space of unparametrised $J_0$-holomorphic cylinders is a two-dimensional manifold.

In particular, it follows that no simply covered cylinder satisfies the tangency condition $\C\mathbf{v} \subset T_pL$ after a generic perturbation of $J_0$, given that the point $p \in L$ and the tangency condition both were chosen generically.

There are now two possibilities for the limit pseudoholomorphic building produced above:
\begin{itemize}
\item The component $C_0$ of the limit building is a multiple cover of a $J_0$-holomorphic cylinder in $T^*L$ branched at $p$; or
\item The component $C_0$ of the limit building has a non-zero tangency to $\C\mathbf{v}$.
\end{itemize}
In either of the two cases, the component $C_0$ can be seen to have at least three punctures.
\end{proof}
\begin{rem} There are two alternative approaches to proving Proposition \ref{prop:three} which do not require a perturbation of the almost complex structure $J_0$ on $T^*L$:
\begin{enumerate}
\item The space of simple $J_0$-holomorphic cylinders can be described explicitly and moreover seen to come in a two-dimensional family (as expected); see \cite{LagIsoTori};
\item Since the component $C_0$ is the limit of embedded pseudoholomorphic spheres, it can be shown to be a (possibly trivial) branched cover of an embedded punctured pseudoholomorphic sphere. Given that the underlying simply covered sphere is a cylinder, Wendl's automatic transversality theorem in \cite{Auttrans} shows that it is transversely cut out.
\end{enumerate}
\end{rem}

\subsection{A further analysis of the obtained pseudoholomorphic building}
\label{sec:further}
Using the assumption that $D_\infty$ is $J^\infty$-holomorphic (see Section \ref{sec:neckstretch}), there is a unique punctured $J^\infty$-holomorphic sphere $A_\infty$ in the top level of the above building which intersects the line $D_\infty$ at infinity transversely in precisely one point. Here we rely on positivity of intersection \cite{LocalCurve} together with $[D_\infty]\bullet[D_\infty]=1$. We proceed to establish certain properties of the punctured sphere $A_\infty$ passing through the line at infinity that will be needed for the proof.
\begin{prop}
\label{prop:area}
The punctured $J^\infty$-holomorphic sphere $A_\infty$ intersecting the line at infinity is a simply covered plane (i.e.~a one-punctured sphere) of symplectic area $0<\int_{A_\infty}\omega_{\OP{FS},{1+\delta}} \le \pi \delta$.
\end{prop}
\begin{proof}
The area estimate follows from Proposition \ref{prop:three}, since $\int_{A_i}\omega_{\OP{FS},{1+\delta}}\ge \pi/2$, $i=1,2$, by assumption, and since the total symplectic area of the components in the top level is equal to $\pi(1+\delta)$ by topological reasons. To that end, we must use the fact that every punctured sphere in the top level has positive symplectic area.

We continue by arguing that $A_\infty$ has a single puncture. Otherwise, using Lemma \ref{lem:noplanes}, in addition to the two planes $A_1,A_2 \subset \CP^2 \setminus L$ established by Proposition \ref{prop:three}, $A_\infty$ must be connected to an additional $J^\infty$-holomorphic plane $A \subset \CP^2 \setminus L$ disjoint from $D_\infty$. By the assumption that $L$ is extremal, the symplectic area of this plane satisfies $\int_A\omega_{\OP{FS},1+\delta} \ge \pi/2$. Using an area argument as above, the existence of these three planes in $\CP^2 \setminus L$ disjoint from $D_\infty$ now leads to a contradiction.
\end{proof}
Similarly, we also obtain
\begin{prop}
\label{prop:nobubbling}
The connected component $\mathcal{M}$ containing $A_\infty$ of the moduli space of pseudoholomorphic planes cannot bubble given that $\pi \delta < \pi/2$. It thus follows by \cite{CompSFT} that this component is compact.
\end{prop}
\begin{proof}
By Lemma \ref{lem:noplanes}, a hypothetical limit pseudoholomorphic building of a sequence of planes in the moduli space must consist of at least two $J^\infty$-holomorphic components in the top level $\CP^2 \setminus L$. Positivity of intersection moreover implies that exactly one of these components intersects the line at infinity $D_\infty$.

Assuming that the broken pseudoholomorphic building exists, we pick one such component $A \subset \CP^2 \setminus L$ which is disjoint from $D_\infty$. Since every component in $\CP^2 \setminus L$ has positive symplectic area, with total sum equal to at most $\pi\delta$ by Proposition \ref{prop:area}, it follows that $0<\int_A\omega_{\OP{FS},1+\delta} < \pi \delta$. This however contradicts the assumption that $L$ is extremal. Namely, the compactification of $A$ produces a class in $H_2(D^4,L)$ of the same symplectic area. Observe that even in the case when $A$ itself is not a plane, since $D^4$ is simply connected, this class can be represented by a disc.
\end{proof}

For the embeddedness properties shown in the following proposition, we need to make heavy use of positivity of intersection. The proof does hence not generalise to higher dimensions.
\begin{prop}
\label{prop:embed}
All planes in the connected component $\mathcal{M}$ containing $A_\infty$ of the moduli space of pseudoholomorphic planes are embedded.
\end{prop}
\begin{proof}
Positivity of intersection \cite{LocalCurve} together with the nature of convergence \cite{CompactnessPunctured} implies that every component in the limit building is a (possibly trivial) branched cover of an \emph{embedded} punctured pseudoholomorphic sphere. Here we have used the fact that the spheres in the sequence constructed by Gromov's Theorem \ref{thm:gromov} all are embedded (this again follows from positivity of intersection), together with the fact that a discrete intersection point of a singularity would contribute positively to the local self-intersection index defined by D. McDuff in \cite{LocalCurve}.

Since $A_\infty$ is intersects $D_\infty$ with intersection number one, it is simply covered and hence embedded. The properties of the local self-intersection index in \cite{LocalCurve} moreover implies that all planes in the component $\mathcal{M}$ must be embedded. To that end, observe that these planes necessarily are embedded outside of a fixed compact subset by \cite[Corollary 2.6]{Siefring:Relative} together with Proposition \ref{prop:nobubbling}.
\end{proof}

\begin{prop}
\label{prop:index}
After a perturbation of $J^\infty$ in $U\setminus D_\infty \subset \CP^2$, where $U$ is an arbitrarily small neighbourhood of the line at infinity $D_\infty$, we may assume the following. The $J^\infty$-holomorphic plane $A_\infty$ is of odd and positive Fredholm index. (This Fredholm index denotes the expected dimension of the component of the moduli space $\mathcal{M} \ni A_\infty$ of \emph{unparametrised} pseudoholomorphic planes for which the asymptotic Reeb orbit is allowed to vary.)
\end{prop}
\begin{proof}
Let $\Gamma$ be the one-dimensional family of periodic Reeb orbits containing the asymptotic of $A_\infty$, and let $\gamma$ be the oriented geodesic on $L$ corresponding to this asymptotic orbit. (The orientation of $\gamma$ is induced by the flow of the Reeb vector field.) Note that the oriented boundary of the compactification of $A_\infty$ to a disc in $X$ is equal to the geodesic $-\gamma$ on $L$, i.e.~the geodesic endowed with the opposite orientation.

Using \cite[Formula (3)]{PuncturedHolomorphic} we can express the sought Fredholm index as
\[\OP{index}(A_\infty)=-1+2c_1(A_\infty)-(\OP{RS}(\Gamma)-1/2),\]
where $\OP{RS}(\Gamma)$ denotes the Robbin-Salamon index defined in \cite[Remark 5.4]{MaslovIndex}.

Using an appropriate trivialisation of the complex determinant bundle of $T\CP^2$ in order to compute the above relative first Chern class and Robbin-Salamon index, we obtain the identity
\[-(\OP{RS}(\Gamma)-1/2)=\mu_L(-\gamma)\]
by \cite[Lemma 2.1]{PuncturedHolomorphic} (this was originally proven in \cite{NewObstruction}). Here $\mu_L(\gamma)$ denotes the Maslov index of the closed geodesic $\gamma$ on $L \subset (\C^2,\omega_0)$ computed using the canonical trivialisation of $T\C^2$. In particular, since $\mu_L(\gamma)$ is even by the orientability of $L$, we conclude that $\OP{index}(A_\infty)$ must be \emph{odd}.

Since $A_\infty$ is simply covered by Proposition \ref{prop:area}, a standard transversality argument implies the following (see \cite[Section 3.4]{JHolCurves}). After a perturbation of $J^\infty$ supported in $U \setminus D_\infty$, where $U \subset \CP^2$ is an arbitrarily small neighbourhood of $D_\infty$, the solution $A_\infty$ may be assumed to be transversely cut out. Its index can thus be assumed to be \emph{non-negative}, since it is equal to the dimension of the moduli space containing it.
\end{proof}

\begin{rem}
Utilising the formula for how the Fredholm indices of the components of a building compare to the Chern number of the building, one can even show that the Fredholm index of $A_\infty$ is equal to one; see \cite{LagIsoTori}. This fact will however not be needed.
\end{rem}

\begin{rem}
Even if the pseudoholomorphic plane $A_\infty$ is simply covered, it is possible that its unique puncture is asymptotic to a multiply covered periodic Reeb orbit. The example of the Chekanov torus in \cite[Corollary A.2]{PuncturedHolomorphic} shows that this case indeed can occur.
\end{rem}

\subsection{A null-homology of $L$ by a chain foliated by planes}
Consider the embedded plane $A_\infty \subset \CP^2 \setminus L$ obtained above which is asymptotic to a Reeb orbit in $\Gamma$, and let $\mathcal{M}$ denote the connected component containing $A_\infty$ of the moduli space of \emph{unparametrised} $J^\infty$-holomorphic planes. For a fixed Reeb orbit $\gamma \in \Gamma$ we let $\mathcal{M}_\gamma \subset \mathcal{M}$ denote the subspace of those planes having $\gamma$ as its asymptotic orbit.

By Propositions \ref{prop:embed} and \ref{prop:index}, the equality in \cite[Remark 1.2]{Auttrans} is satisfied for the planes in both $\mathcal{M}$ and $\mathcal{M}_\gamma$. The automatic transversality result \cite[Theorem 1]{Auttrans} thus applies both of the moduli spaces $\mathcal{M}$ and $\mathcal{M}_\gamma$. These two results combine to show that the component $\mathcal{M}$ is transversely cut out and, moreover, that the asymptotic evaluation map
\[\OP{ev}_\infty \colon \mathcal{M} \to \Gamma \simeq S^1\]
is a submersion. Since $\mathcal{M}$ is compact by Proposition \ref{prop:nobubbling}, we have shown that this moduli space is the total space of a smooth locally trivial fibre bundle over $S^1$ with compact fibre diffeomorphic to $\OP{ev}^{-1}_\infty(\gamma)=\mathcal{M}_\gamma$.

Let $\varphi \in \OP{Diff}(\mathcal{M}_\gamma)$ be the clutching function for this fibre bundle. Consider the pull-back bundle $\widetilde{\mathcal{M}} \to S^1$ induced by the $n$-fold covering $S^1 \to S^1$, which thus can be constructed using the clutching function $\varphi^n$. Since $|\pi_0(\mathcal{M}_\gamma)| < \infty$ holds by compactness, taking $k := |\pi_0(\mathcal{M}_\gamma)| \ge 1$ it follows that $\widetilde{\mathcal{M}}$ admits a section $\sigma_{n} \colon S^1 \to \widetilde{\mathcal{M}}$ whenever $k!|n$.

For any contractible open subset $U \subset \mathcal{M}$, one can define evaluation maps
\[ \OP{ev} \colon U \times \C \to \CP^2 \setminus L\]
for the planes in $U$. These evaluation maps obviously depend on the choice of a holomorphic parametrisation of the domain, i.e.~the plane $\C$. Recall that any two such parametrisations of $\C$ are related by a biholomorphism $z \mapsto az+b$ for some $a\in \C^*$ and $b \in \C$.

The constant $b \in \C$ together with the modulus $|a|>0$ of $a$ both constitute contractible choices when fixing a parametrisation. We are left with the choice of $\arg a \in S^1 \subset \C^*$. Let $m \ge 1$ denote the multiplicity of the Reeb orbits in $\Gamma$, and fix a smoothly depending choice of starting point for each Reeb orbit in $\Gamma$. We can make $\arg a$ uniquely determined up to multiplication by a power of $e^{i2\pi /m} \in \C^*$ by requiring the parametrisation to be asymptotic to the chosen starting point of its asymptotic Reeb orbit as $\mathfrak{R}(z)=x \to +\infty$ along the positive real line. (This uses the asymptotic convergence to a Reeb orbit satisfied by a finite energy plane; see e.g.~\cite{CompSFT}.)

By the above we can fix smoothly depending choices of parametrisations of the $S^1$-family of planes $\sigma_{k!m!}$. Namely, the (homotopy classes of) parametrisations of these planes form a bundle over $S^1$ with fibre $\Z/\Z m$ and clutching function $\hat{\varphi}^{m!}=\id_{\Z / \Z m}$. 

In conclusion, it is possible to construct a globally defined evaluation map for the $S^1$-family of planes $\sigma_{k!m!} \colon S^1 \to \widetilde{\mathcal{M}}$. Using the property of asymptotic convergence to Reeb orbits satisfied by finite energy planes (again see e.g.~\cite{CompSFT}), the latter evaluation map can be compactified to produce a continuous map
\[ (S^1 \times D^2,S^1\times S^1) \to (\CP^2,L), \]
whose restriction to the boundary is a map $S^1 \times S^1 \to L$ of degree $k m l >0$ for some $l>0$. For the latter property we use the fact established above that the asymptotic evaluation map is a submersion, together with the fact that each plane is asymptotic to a closed geodesic on $L$ for the flat metric.

Note that the image of each $\{\theta\} \times D^2$ under the above map is $J^\infty$-holomorphic away from the boundary. Let us choose $\delta>0$ sufficiently small in order for $\hbar > \pi\delta>0$ to be satisfied, where $\hbar$ is the constant given by \ref{M}. Since the linking number of $L$ and the fibre $S_0 \subset D^4 \setminus L$ constructed in Section \ref{sec:proof} is non-zero inside $\CP^2$, namely there is a disc with boundary equal to $S_0$ which intersects $L$ in a single transverse point $p \in L$, the above chain must intersect $S_0$. However, the monotonicity property \ref{M} holds (recall that $J^\infty=J_{U_0}$ in the neighbourhood $U_0 \supset S_0$; see Section \ref{sec:neckstretch}), thus leading to the sought contradiction with Proposition \ref{prop:area}. See Figure \ref{fig:disc} for an illustration.

\begin{figure}[htp]
\begin{center}
\vspace{3mm}
\labellist
\pinlabel $\color{red}S_0$ at 69 87
\pinlabel $\color{blue}D$ at 78 66
\pinlabel $L$ at 106 75
\pinlabel $\color{blue}\infty$ at 153 70
\pinlabel $(\CP^2,\omega_{\OP{FS},1+\delta})$ at 166 21
\pinlabel $B^4$ at 74 20
\pinlabel $P$ at 50 40
\endlabellist
\includegraphics{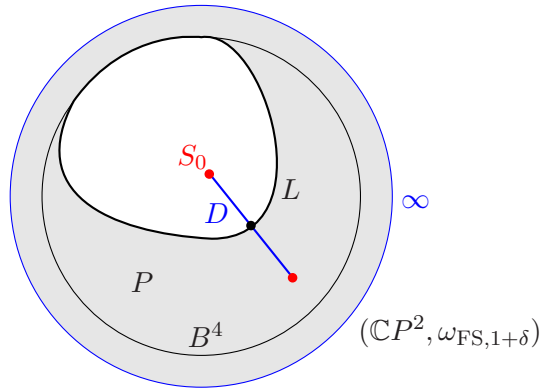}
\caption{The disc $D$ with boundary equal to $S_0$ intersects $L$ transversely in a single point. A null-homology of $L$ must thus intersect $S_0 \subset \CP^2 \setminus L$. In particular, one of the planes $P \in \mathcal{M}$ must intersect $S_0$.}
\label{fig:disc}
\end{center}
\end{figure}

\def\cprime{$'$} \def\cprime{$'$} \def\cprime{$'$} \def\cprime{$'$}

%\bibliographystyle{plain}
%\bibliography{references}

\end{document}